\documentclass[12pt]{amsart}

\usepackage{amsfonts,amssymb,cleveref}
\usepackage{amsmath}
\usepackage{amsthm}

\newtheorem{thm}{Theorem}[section]
\newtheorem*{thm*}{Theorem}
\newtheorem*{conj*}{Conjecture}

\newtheorem{prop}[thm]{Proposition}

\theoremstyle{remark}

\theoremstyle{definition}

\newtheorem{defn}[thm]{Definition}

\DeclareMathOperator{\Cay}{Cay}
\DeclareMathOperator{\Aut}{Aut}
\DeclareMathOperator{\Dih}{Dih}

\begin{document}

\title{Two families of graphs that are Cayley on nonisomorphic groups}

\author{Joy Morris} 
\address{Department of Mathematics and Computer Science\\
University of Lethbridge\\
Lethbridge, AB. T1K 3M4}
\email{joy.morris@uleth.ca}\thanks{This work was supported by the Natural Science and Engineering Research Council of Canada (grant RGPIN-2017-04905). The second author worked on this project as a summer research experience supported out of this grant.}
\author{Josip Smol\v{c}i\'{c}} 
\address{Department of Mathematics and Computer Science\\
University of Lethbridge\\
Lethbridge, AB. T1K 3M4}
\email{josip.smolcic@uleth.ca}

\begin{abstract}
A number of authors have studied the question of when a graph can be represented as a Cayley graph on more than one nonisomorphic group. The work to date has focussed on a few special situations: when the groups are $p$-groups; when the groups have order $pq$; when the Cayley graphs are normal; or when the groups are both abelian. In this paper, we construct two infinite families of graphs, each of which is Cayley on an abelian group and a nonabelian group. These families include the smallest examples of such graphs that had not appeared in other results.
\end{abstract}

\maketitle

\section{Introduction}

A Cayley graph $\Cay(G,S)$ on a group $G$ with connection set $S$, is the graph whose vertices are the elements of $G$, with two vertices $g_1$ and $g_2$ adjacent if and only if $g_2=sg_1$ for some $s \in S$. In order to ensure that this is a graph rather than a directed graph, we must require that $S=S^{-1}$; that is, $S$ is closed under inversion; if we omit this condition, we obtain digraphs (and an arc from $g_1$ to $g_2$ rather than an edge between them). Conventionally we also generally assume that the identity $e$ of $G$ is not in $S$; this avoids having loops at every vertex. Cayley graphs and digraphs are a major area of study, as their symmetries lead to many useful properties in the networks they represent.

It is well known (first observed by Sabidussi) that a (di)graph can be represented as a Cayley (di)graph on the group $G$ if and only if its automorphism group contains a subgroup isomorphic to $G$ in its regular action. However, a particular representation of a Cayley (di)graph may not be its only representation, either on a fixed group, or on different groups. Sometimes a particular representation may be more useful for practical purposes  than a different representation, so it is of interest to understand all possible representations.

The so-called ``Cayley Isomorphism" (CI) problem studies whether or not all representations for a given Cayley graph on some fixed group $G$ can be determined purely algebraically. It is therefore a large part of the question of when a Cayley graph on a group $G$ is isomorphic to another Cayley graph on the same group $G$ (or, equivalently, when there are two distinct regular subgroups isomorphic to $G$ in the automorphism group of the graph). The CI problem has been extensively studied by many researchers. For example, the papers \cite{babai,Dobson1995,Dobson2002,LiLP,Muz,Muz1} amongst many others, and the survey article \cite{Lisurvey} all deal with this question.

The question of when a Cayley graph on $G$ can be represented as a Cayley graph on some nonisomorphic group $H$ has also received some attention. Joseph in 1995 \cite{Joseph1995}  determined necessary and sufficient conditions for a Cayley digraph of order $p^2$ (where $p$ is prime), to be isomorphic to a Cayley digraph of both groups of order $p^2$ (\cite[Lemma 4]{DobsonW2002} provides a group theoretic version of this result).  The first author \cite{Morris1996,Morris1999} subsequently extended this result and determined necessary and sufficient conditions for a Cayley digraph of the cyclic group of order $p^k$, $k\ge 1$ and $p$ an odd prime, to be isomorphic to a Cayley digraph of some other group of order $p^k$.  The equivalent problem for $p=2$ (when both groups are abelian) was solved by Kov\'acs and Servatius \cite{KovacsS2012}.  In these cases, graphs that could be represented on both groups are all ``wreath" (or ``lexicographic") products, and their automorphism groups are significantly larger than the number of vertices. In contrast, when neither group is cyclic, \cite{MMV} shows that it is often possible to find Cayley digraphs that can be represented on two nonisomorphic $p$-groups (one abelian and the other not) whose automorphism group is only slightly larger than the original groups.

Digraphs of order $pq$ that are Cayley graphs of both groups of order $pq$, where $q\mid (p-1)$ and $p,q$ are distinct primes were determined by Dobson in \cite[Theorem 3.4]{Dobson2006a}.  Maru\v si\v c and the first author studied the question of which normal circulant graphs of square-free order are also Cayley graphs of a nonabelian group \cite{MarusicM2005}. Some of the graphs in our families fall into each of these categories, but neither of our families is limited to square-free orders.

\section{The families}

The first of these families may be known to researchers, but to the best of our knowledge no proof has previously appeared in the literature. A circulant graph is a Cayley graph on a cyclic group, and we use $D_k$ to denote the dihedral group of order $2k$.

\begin{prop}Let $\Gamma$ be a circulant graph on $n = 2k$ vertices. Then $\Gamma$ is a Cayley graph on $D_k$ and $C_n$.
\end{prop}

\begin{proof} 
  Let $\Gamma = \Cay(C_n, S)$, where $S \subset C_n$ is closed under inverses, and $C_n=\langle c\rangle$. By assumption, $\Gamma$ is a Cayley graph on $C_n$. 
  
  We must show that $\Gamma$ is also a Cayley graph on $D_k$. We do this by finding a regular subgroup of $\Aut(\Gamma)$ that is isomorphic to $D_k$. 
  
Define $\alpha$ by $\alpha(z) = z c^2$ and $\beta(z) = z^{-1}c^{-1}$ for $z \in V(\Gamma)=C$.
  
  We first show that $\alpha$ and $\beta$ are automorphisms. For every $u, v \in V(\Gamma)$ with $u \sim v$, there exists $s \in S$ such that $su = v$. It is not hard to see that 
  $$s \alpha(u) = suc^2 = vc^2 = \alpha(v).$$ 
  Also, since $S$ is closed under inverses and $u$ and $s$ are both elements of the abelian group $C$, we have
  $$s^{-1} \beta(u) =s^{-1}u^{-1}c^{-1} = (us)^{-1}c^{-1} = v^{-1}c^{-1} = \beta(v)$$ 
  as desired.
  
Since $n=2k$ is the order of $c$, it is clear that $\alpha$ has order $k$. Also $\beta^2 (z) = \beta(z^{-1}c^{-1}) = (z^{-1}c^{-1})^{-1}c^{-1}=czc^{-1} = z$, thus $\beta$ has order $2$. Finally, $$\beta^{-1}\alpha \beta (z) = \beta\alpha (z^{-1}c^{-1}) = \beta(z^{-1}c)=zc^{-2} = \alpha^{-1}(z),$$ so $\beta$ inverts $\alpha$. We conclude that $\langle \alpha, \beta\rangle \cong D_k$ is a regular subgroup of $\Aut(\Gamma)$, so $\Gamma$ is a Cayley graph on $D_k$.
\end{proof}

The second family has slightly more restrictions, but is at the same time potentially more interesting. To understand it, we must define the family of generalised dihedral groups.

\begin{defn}
Let $A$ be an abelian group. Define the group $\Dih(A,x)=\langle A, x\rangle$, where $x^2=1$ and $x^{-1}ax=a^{-1}$ for every $a \in A$.  
\end{defn}

In the special case where $A$ is cyclic, this is the usual dihedral group. Notice that the group $\Dih(A,x)$ is abelian if and only if $A$ is an elementary abelian $2$-group, in which case $\Dih(A,x)$ is the elementary abelian $2$-group whose rank is one higher than the rank of $A$.

\begin{thm}
Let $A$ be an abelian group, and let $D=\Dih(A,x)$ be the corresponding generalised dihedral group. Let $S \subseteq D$ be closed under inversion, and let $\Gamma = \Cay (D, S)$.

Suppose that there is some $y \in xA$ such that $ya \in S \cap xA$ if and only if $ya^{-1} \in S \cap xA$.  Then $\Aut(\Gamma)$ has a subgroup isomorphic to $A \times C_2$, so $\Gamma$ is also a Cayley graph on the abelian group $A\times C_2$.
\end{thm}

\begin{proof}
First note that if $A$ is an elementary abelian $2$-group, then $\Dih(A,x) \cong A \times C_2$ so there is nothing to prove.

For every $a \in A$, define the map $\alpha_a$ on the vertices of $\Gamma$ by $\alpha_a (z) = za$, and define the map $\beta$ by $\beta (z) = yz$ for all $z \in V(\Gamma)=D$.
  Let $H = \langle \alpha_{a}, \beta : a \in A\rangle$. We claim that $H \cong A \times C_2$ is a subgroup of $\Aut(\Gamma)$.
  
  First we show that $H \cong A \times C_2$. It should be clear that $\langle \alpha_a: a \in A\rangle \cong A$. Furthermore, since $y \in xA$ we have $y=xa$ for some $a \in A$, so $y^2=xaxa=a^{-1}a=1$, meaning that $\beta$ has order $2$. It remains only to show that $H$ is abelian. Again, since $A$ is abelian, we really only need to show that $\beta$ commutes with every $\alpha_a$. This is easy, since $$\beta\alpha_a(z)=\beta(za)=yza=\alpha_a(yz)=\alpha_a\beta(z).$$

  For the remainder of the proof, we must show that $H$ consists of automorphisms of $\Gamma$. Let $u, v \in V(\Gamma)$ where $u \sim v$, so there is some $s \in S$ such that $v=su$. Now let $\alpha_a \in H$. We have
  $$s \alpha_a(u) = sua = va = \alpha_a(v),$$
  so $\alpha_a(u) \sim \alpha_a(v)$ if and only if $u \sim v$, meaning that $\alpha_a$ is an automorphism of $\Gamma$.
  
  To show that $\beta$ is also an automorphism of $\Gamma$, we will require the extra conditions we assumed for $S$: that $S$ is inverse-closed (which is necessary for $\Gamma$ to be a graph rather than a digraph) and also that $ya \in S$ if and only if $ya^{-1} \in S$. We will also need the observation that for every $ a \in A$, $y^{-1}ay=yay=a^{-1}$; this follows immediately from the definitions of $y$ and $x$ and the fact that $A$ is abelian.
  
  Again, we take $u, v \in V(\Gamma)$ where $u \sim v$, so there is some $s \in S$ such that $v=su$. We deal separately with the possibilities that $s \in A$ or $s \in xA=yA$. 
  
  Suppose first that $s \in A$.
 Since $S$ is closed under inverses
  $$s^{-1} \beta(u) = s^{-1}yu = ysu = yv = \beta(v).$$
  Thus $\beta(u) \sim \beta(v)$ if and only if $u \sim v$, meaning that $\beta$ is an automorphism of $\Gamma$.
  
Now suppose that $s \in xA=yA$, say $s=yb$ where $b \in A$. Then $yb^{-1}$ is also in $S$, and
  $$yb^{-1} \beta(u) = yb^{-1}yu = y(yb)u =ysu=yv =\beta(v).$$
  Thus $\beta(u) \sim \beta(v)$ if and only if $u \sim v$, meaning that $\beta$ is an automorphism of $\Gamma$.
\end{proof}

In the case where $A$ is not an elementary abelian $2$-group, we have shown that such graphs are Cayley graphs on both the abelian group $A \times C_2$ and the nonabelian group $\Dih(A,x)$, which are nonisomorphic. It is easy to construct examples of graphs that satisfy our restriction on the connection set; for example, any connection set that contains exactly one element of $xA$ will have this property.

It would be nice to completely characterise the Cayley graphs on $\Dih(A,x)$ that are also Cayley on $A \times C_2$. This would, however, require a fairly deep understanding of the full automorphism group of any such graph (for example, whether or not the cosets of $A$ are blocks of imprimitivity for the automorphism group will be important) that is beyond the scope of this project.


\begin{thebibliography}{99}
\bibitem{babai} L.~Babai, Isomorphism problem for a class of point-symmetric structures, \textit{Acta Math. Acad. Sci.
Hungar.} \textbf{29} (1977), 329--336.

\bibitem{Dobson1995} E.~Dobson, Isomorphism problem for Cayley graphs of ${\mathbb Z}\sp 3\sb
	p$, \textit{Discrete Math.} \textbf{147} (1995), 87--94.
	
\bibitem{Dobson2002}
Edward Dobson, \emph{On the {C}ayley isomorphism problem}, Discrete Math.
  \textbf{247} (2002), no.~1-3, 107--116. \MR{MR1893021 (2003c:05106)}
  
  \bibitem{Dobson2006a}
Edward Dobson, \emph{Automorphism groups of metacirculant graphs of order a product
  of two distinct primes}, Combin. Probab. Comput. \textbf{15} (2006), no.~1-2,
  105--130. \MR{MR2195578 (2006m:05108)}

  
  \bibitem{DobsonW2002}
Edward Dobson and Dave Witte, \emph{Transitive permutation groups of
  prime-squared degree}, J. Algebraic Combin. \textbf{16} (2002), no.~1,
  43--69. \MR{MR1941984 (2004c:20007)}


\bibitem{Joseph1995}
Anne Joseph, \emph{The isomorphism problem for {C}ayley digraphs on groups of
  prime-squared order}, Discrete Math. \textbf{141} (1995), no.~1-3, 173--183.
  \MR{1336683 (96e:05071)}

\bibitem{KovacsS2012}
Istv{\'a}n Kov{\'a}cs and Mary Servatius, \emph{On {C}ayley digraphs on
  nonisomorphic 2-groups}, J. Graph Theory \textbf{70} (2012), no.~4, 435--448.
  \MR{2957057}


\bibitem{Lisurvey}C.~H.~Li,  On isomorphisms of finite Cayley graphs--a survey, \textit{Discrete Math.} \textbf{256} (2002), 301--334.

\bibitem{LiLP}C.~H.~Li,~Z.~P.~Lu, P.~Palfy, Further restrictions on the structure of finite CI-groups, \textit{J. Algebr. Comb. }\textbf{26} (2007), 161--181.

\bibitem{MarusicM2005}
Dragan Maru{\v{s}}i{\v{c}} and Joy Morris, \emph{Normal circulant graphs with
  noncyclic regular subgroups}, J. Graph Theory \textbf{50} (2005), no.~1,
  13--24. \MR{2157535 (2006c:05073)}
  
\bibitem{MMV}
Luke Morgan, Joy Morris, and Gabriel Verret, \emph{Digraphs with small automorphism groups that are Cayley on two nonisomorphic groups} The Art of Discrete and Applied Mathematics \textbf{3} (2020),  \#P1.01.

\bibitem{Morris1996}
Joy Morris, \emph{Isomorphic {C}ayley graphs on different groups}, Proceedings
  of the {T}wenty-seventh {S}outheastern {I}nternational {C}onference on
  {C}ombinatorics, {G}raph {T}heory and {C}omputing ({B}aton {R}ouge, {LA},
  1996), vol. 121, 1996, pp.~93--96. \MR{MR1431979 (97k:05102)}

\bibitem{Morris1999}
\bysame, \emph{Isomorphic {C}ayley graphs on nonisomorphic groups}, J. Graph
  Theory \textbf{31} (1999), no.~4, 345--362. \MR{MR1698752 (2000e:05085)}



\bibitem{Muz}M.~Muzychuk, On the isomorphism problem for cyclic combinatorial objects, \textit{Discrete Math.} \textbf{197/198} (1999), 589--606.

\bibitem{Muz1}M.~Muzychuk, A solution of the isomorphism problem for circulant graphs, \textit{Proc. London Math. Soc. }\textbf{88} (2004), 1--41.

\end{thebibliography}
\end{document}